\documentclass[11pt,reqno]{amsart}
\usepackage[english]{babel}
\usepackage{amsmath,amssymb,amsfonts,amsthm,latexsym,graphicx} 
\usepackage{diagrams}
\usepackage{cite}
\usepackage{extpfeil}
\usepackage{mathtools}
\usepackage[pagewise]{lineno}
\numberwithin{equation}{section}
\newextarrow{\xbigtoto}{{20}{20}{20}{20}}
   {\bigRelbar\bigRelbar{\bigtwoarrowsleft\rightarrow\rightarrow}}
\theoremstyle{definition}
\newtheorem{thm}{Theorem}[section]

\newtheorem{lem}[thm]{Lemma}
\newtheorem{exa}[thm]{Example}
\newtheorem{prop}[thm]{Proposition}
\newtheorem{defi}[thm]{Definition}
\newtheorem{rem}[thm]{Remark}
\newtheorem{note}[thm]{Notation}

\DeclareMathOperator{\red}{\mathrm{red}}

\DeclareMathOperator{\msp}{\mathrm{Spec}}

\DeclareMathOperator{\pr}{\mathrm{pr}}

\DeclareMathOperator{\Hom}{\mathrm{Hom}}
\DeclareMathOperator{\Spec}{\mathrm{Spec}}

\DeclareMathOperator{\mo}{\mathcal{O}}

\newcommand{\mr}[1]{\mathrm{#1}}
\newcommand{\mb}[1]{\mathbb{#1}}
\newcommand{\mc}[1]{\mathcal{#1}}
\newcommand{\ov}[1]{\overline{#1}}

\begin{document}

\title{Descent theory of simple sheaves on $C_1$-fields}

\author[A. Dan]{Ananyo Dan}

\address{Basque Centre for Applied Mathematics, Alameda de Mazarredo 14,
48009 Bilbao, Spain}

\email{adan@bcamath.org}

\author[I. Kaur]{Inder Kaur}

\address{Pontifícia Universidade Cat\'{o}lica do Rio de Janeiro (PUC-Rio), R. Marqu\^{e}s de S\~{a}o Vicente, 225 - G\'{a}vea, Rio de Janeiro - RJ, 22451-900, Brasil }

\email{inder@mat.puc-rio.br}

\subjclass[2010]{Primary $12$G$05$, $16$K$50$, $14$D$20$, $14$J$60$, Secondary $14$L$24$, $14$D$22$}


\date{January $31$, $2019$}

\begin{abstract}
Let $K$ be a $C_1$-field of any characteristic and 
$X$ a projective variety over $K$. In this article we prove that for a finite Galois extension $L$ of $K$, a simple sheaf with covering datum on $X \times_K L$ descends to a simple sheaf on $X$. 
As a consequence, we show that there is a $1-1$ correspondence between the set of geometrically stable sheaves on $X$ with fixed Hibert polynomial $P$ and the set of $K$-rational points of the 
corresponding moduli space.
\end{abstract}

\maketitle

\section{Introduction}
Let $f:Y \to X$ be a morphism of schemes and $\pr_i:Y \times_X Y \to Y$ the natural projection morphisms. For any quasi-coherent sheaf $\mc{F}$
on $Y$, a \emph{covering datum of} $\mc{F}$ is an isomorphism $\phi:\pr_1^*\mc{F} \to \pr_2^*\mc{F}$. If $f$ is faithfully flat and quasi-compact then the functor
from the category of quasi-coherent sheaves on $X$ to the category of quasi-coherent sheaves on $Y$ with covering datum, induced by pull-back by $f$, is fully faithful 
(see \cite[\S $6.1$, Proposition $1$]{bn}). 
Let us consider the simple case when $X$ is a projective variety over a field $K$, $L$ a finite Galois extension of $K$ and $Y := X \times_K \msp(L)$
is the base change of $X$ by the field extension $L$ of $K$. For a coherent sheaf $\mc{E}_L$ on $Y$, descent theory gives a 
criterion for when there exists a coherent sheaf $\mc{E}$ on $X$ such that $f^*\mc{E} \cong \mc{E}_L$. 
The descent datum associated to $\mc{E}_L$ consists of a covering datum satisfying a cocycle condition. In general, it is not possible 
to associate to $\mc{E}_L$ a descent datum, for example, if $\mc{E}_L$ is the structure sheaf of an $L$-point on $Y$ that does not come from a $K$-point of $X$, then
there does not exist a descent datum associated to $\mc{E}_L$.

\vspace{0.2 cm}
Recall that a field is called $C_1$ (pseudo-algebraically closed) if any degree $d$ polynomial in $n$ variables with coefficients in the field and $n>d$ has a non-trivial solution. 
These fields have been studied by Tsen \cite{T}, Lang \cite{lang1}, Chevalley \cite{C}, Greenberg \cite{G}, Grabber-Harris-Starr \cite{GHS} and many others (see \cite[\S $2.1$]{ind}
for a short introduction to $C_1$ fields).
In this article, we consider the case when $X$ is a projective variety over a $C_1$-field $K$, $L$ is a finite Galois extension of $K$ and $Y:=X \times_K \Spec(L)$. 
We call a simple (resp. geometrically stable) sheaf $\mc{E}_L$ on $Y$, $K$-\emph{simple} (resp. $K$-\emph{geometrically stable}) 
if one can associate a covering datum to $\mc{E}_L$ (see Definition \ref{elw15}). 

We prove that:

\begin{thm}\label{in1}
 The natural functor from  the category $\mathfrak{S}_X$ of simple sheaves on 
 $X$ to the category $\mathfrak{SK}_Y$ of $K$-simple sheaves on $Y$, defined by pull-back of sheaves via $f$,
is an equivalence of categories. 
\end{thm}

 Clearly, any coherent sheaf $\mc{E}_L$ with covering datum on $Y$ descends to $X$ 
if the covering datum satisfies the cocycle condition.
However, we show that the cocycle condition is indeed satisfied if $\mc{E}_L$ is $K$-simple and 
$H^2(K,\mb{G}_m)$ vanishes (see proof of Proposition \ref{c8}). Since $K$ is a $C_1$-field, the cohomology group $H^2(K,\mb{G}_m)$ is
trivial (see \cite[Proposition $6.2.3$]{gill}). We then give an application of Theorem \ref{in1}. 
Fix $P$ the Hilbert polynomial of a coherent sheaf on $X$ with rank coprime to degree. 
Denote by $M_X(P)$ the moduli scheme of geometrically stable sheaves on $X$ with Hilbert polynomial $P$. 
Recall, a point $x \in X$ is called a $K$-\emph{rational point} if the corresponding residue field is contained in $K$.
We show:

\begin{thm}\label{elw11} 
There is a $1-1$ correspondence between the set of $K$-rational points of $M_X(P)$ and the set of 
geometrically stable sheaves on $X$ with Hilbert polynomial $P$. 
\end{thm}

We note that a similar result to Theorem \ref{elw11} for stacks has been proven by Kraschen and Lieblich in  \cite[Proposition $1.1.5$]{KL}. However their proof uses the theory of gerbes. In contrast our proof uses only basic algebraic geometry.   

In \S \ref{sec4} we give possible applications of the above theorems to the existence of rational points ($C_{1}$-conjecture due to Lang, Manin and Koll\'{a}r)
and index of varieties.

\vspace{0.2 cm}
\emph{Acknowledgements} 
The first author is currently supported by ERCEA Consolidator Grant $615655$-NMST and also
by the Basque Government through the BERC $2014-2017$ program and by Spanish
Ministry of Economy and Competitiveness MINECO: BCAM Severo Ochoa
excellence accreditation SEV-$2013-0323$. 
The second author is funded by a CAPES-PNPD fellowship. A part of this work was done when she was visiting ICTP. She warmly thanks ICTP, the Simons Associateship and Prof. 
Carolina Araujo for making this possible.

\section{Galois action on simple sheaves}

We begin by recalling the basic definitions and results we need. 

\begin{defi}\label{desdefi}
Let $f: S' \to S$ be a morphism of schemes.
Set $S'':= S' \times_{S} S',  S''':= S'\times_{S}S'\times_{S}S'$, $\pr_i:S'' \to S'$ and $\pr_{ij}:S''' \to S''$ the natural projections onto the factors 
with indices $i$ and $j$, for $i<j$, $i,j \in \{1,2,3\}$. 
A \emph{descent datum} on a quasi-coherent sheaf $\mc{F}$ on $S'$ is a covering datum $\phi:\pr_1^*\mc{F} \xrightarrow{\sim} \pr_2^*\mc{F}$ on $\mc{F}$ 
which satisfies the cocycle condition $\pr^{*}_{13}\phi = \pr^{*}_{23}\phi \circ \pr_{12}^{*}\phi$ i.e., $\pr_{13}^*\phi$ coincides with 
the composition:
\[\pr_{13}^*\pr_1^*\mc{F} \cong \pr_{12}^*\pr_1^*\mc{F} \xrightarrow{\pr_{12}^*\phi} \pr_{12}^*\pr_2^* \cong \pr_{23}^*\pr_1^*\mc{F} \xrightarrow{\pr_{23}^*\phi} \pr_{23}^*\pr_2^*\mc{F} \cong \pr_{13}^*\pr_2^*\mc{F},\]
where the three isomorphisms follow from $\pr_1 \circ \pr_{13}=\pr_1 \circ \pr_{12}, \, \pr_1 \circ \pr_{23}=\pr_2 \circ \pr_{12}$ and 
$\pr_2 \circ \pr_{23}=\pr_2 \circ \pr_{13}$, respectively.
\end{defi}

\begin{prop}[{\cite[\S $6.1$, Proposition $1$]{bn}}]\label{blrprop}
Notations as in Definition \ref{desdefi}.
Let $f: S' \to S$ be a faithfully flat and quasi-compact morphism of schemes and $\mc{F}$, $\mc{G}$ be  quasi-coherent $S$-modules and set 
$q:= f \circ \pr_1 = f \circ \pr_2$. 
Then, identifying $q^{*}\mc{F}$ (resp. $q^{*}\mc{G}$) canonically with $\pr_i^{*}(f^{*}\mc{F})$ (resp. $\pr_i^{*}(f^{*}\mc{G})$) for $i = 1,2$, the sequence 
\[ \Hom_{S}(\mc{F},\mc{G}) \xrightarrow{f^{*}} \Hom_{S'}(f^{*}\mc{F},f^{*}\mc{G}) \xbigtoto[\pr^{*}_1]{\pr^{*}_2} \Hom_{S''}(q^{*}\mc{F},q^{*}\mc{G}) \]
\noindent is exact. In other words, the functor $\mc{F} \mapsto  f^{*}\mc{F}$ from quasi-coherent $S$-modules to quasi-coherent $S'$-modules with covering datum 
is fully faithful. 
\end{prop}

We now recall the definition of semi-stable sheaves and simple sheaves.

\begin{defi} 
Let $\mc{E}$ be a coherent sheaf with support of dimension $d$. The Hilbert polynomial $P(\mc{E})(t)$ of $\mc{E}$ can be expressed as (see {\cite[Lemma $1.2.1$]{HL}}) 
\[P(\mc{E})(t) := \chi(\mc{E}\otimes \mc{O}_{X_k}(t)) =  \displaystyle\sum\limits_{i=0}^{d}\alpha_{i}(\mc{E})\dfrac{t^{i}}{i!} 
   \mbox{ for } t>>0.\] 
   The reduced Hilbert polynomial is defined as $P_{\red}(\mc{E})(t) := \dfrac{P(\mc{E})(t)}{\alpha_{d}(\mc{E})}$ . 
The sheaf $\mc{E}$ is called \emph{Gieseker (semi)stable} if for any proper subsheaf 
$\mc{F} \subset \mc{E}$, $P_{\red}(\mc{F})(t) (\leq) < P_{\red}(\mc{E})(t)$ for all $t$ large enough. 
In other words, $\mc{E}$ is (semi)stable if properly included subsheaves
have (strictly) smaller reduced Hilbert polynomials.
\end{defi} 
   
\begin{defi}
A sheaf $\mc{E}$ defined on a projective variety defined over a field $k$ is called \emph{simple} if $\mr{End}(E) \simeq k$. 
\end{defi}

\begin{lem}[{\cite[Corollary $1.2.8$]{huy}}]\label{stabsimp}
If $\mc{E}$ is a stable sheaf on a projective variety defined over an algebraically closed field, say $k$, then $\mr{End}(E) \simeq k$.   
\end{lem}

\begin{note}\label{e1}
 Let $K$ be a $C_1$ field of any characteristic, $X$ a projective variety over $K$. Let $K \subset L$ be an algebraic field extension of $K$ (not necessarily finite). There exist natural morphisms 
 \[\pr_{1,L}:L \to L \otimes_K L, \, \pr_{2,L}:L \to L \otimes_K L \, \mbox{ where } \pr_{1,L}(a)=a \otimes 1 \mbox{ and } \pr_{2,L}(a)=1 \otimes a.\]
 This induces morphisms \[\pr_{i,L}:X_{L \otimes_K L}  \to X_{L} \, \mbox{ for } \, i=1,2.\]
 Denote by $G_L:=\mr{Gal}(L/K)$ the Galois group. For any $\sigma \in G_L$, we denote by 
 \[\sigma:X_{L} \to X_{L}\]
 the induced natural morphism. Moreover, for $\sigma, \tau \in G_L$, we have $\sigma\tau:L \xrightarrow{\tau} L \xrightarrow{\sigma} L$.
 As taking spectrum is contravariant, this induces $(\sigma\tau):X_{L} \xrightarrow{\sigma} X_{L} \xrightarrow{\tau} X_{L}$. Thus, for any coherent sheaf $\mc{E}$
 on $X_{L}$, the pull-back $(\sigma\tau)^*\mc{E}=(\sigma^* \circ \tau^*) \mc{E}$. In the case $L=\ov{K}$, the algebraic closure, denote by $\pr_i:=\pr_{i,\ov{K}}$ and $G:=G_{\ov{K}}$.
  \end{note}
   
 \begin{defi}\label{elw15}
 Recall, a sheaf $\mc{E}$ on $X_L$ is called \emph{geometrically stable} if for any field extension 
 $L'$ of $L$, the sheaf $\mc{E} \otimes_L L'$ is stable over $X_L \times_L \Spec(L')$.
We call a  simple (resp. geometrically stable) sheaf $\mc{E}$ on $X_L$, $K$-\emph{simple} (resp. $K$-\emph{geometrically stable})
  if there exists an isomorphism $\psi:\pr_{1,L}^*\mc{E} \to \pr_{2,L}^*\mc{E}$. In other words, a $K$-simple (resp.  $K$-{geometrically stable}) sheaf
  is a simple (resp. geometrically stable) sheaf on $X_L$ such that one can associate to it a covering datum. 
 \end{defi}
 
We now study the action of the Galois group $G_L$ on a simple sheaf $\mc{E}$ on $X_L$. We observe that in the case $L$ is a finite Galois extension of $K$, $\mc{E}$ descends to $X$ if and only if it is $K$-simple (Theorem \ref{elw7}). 

 \begin{prop}\label{c8}
  Let $\mc{E}$ be a $K$-simple sheaf on $X_{L}$. Then, there exists a collection $(\lambda_\sigma)_{\sigma \in G_L}$ of isomorphisms
  $\lambda_\sigma:\mc{E} \to \sigma^*\mc{E}$ satisfying the cocycle condition:
  \[(\sigma^*\lambda_\tau) \circ \lambda_\sigma=\lambda_{\sigma\tau} \, \mbox{ for any pair } \, \sigma, \tau \in G_L.\]
 \end{prop}
 
 \begin{proof}
  Fix $\sigma \in G_L$. Consider the homomorphism $L \otimes_K L \to L$ defined by $a \otimes b$ maps to $a\sigma(b)$.
  This induces a natural morphism \[p_\sigma:X_{L} \to X_{L \otimes_K L}.\]
  Observe that the morphism $p_\sigma$ has the property that its composition with $\pr_{1,L}$ 
  \[X_{L} \xrightarrow{p_\sigma} X_{L \otimes_K L} \xrightarrow{\pr_{1,L}} X_{L}\]
  is simply the identity map and with $\pr_{2,L}$,
  \[X_{L} \xrightarrow{p_\sigma} X_{L \otimes_K L} \xrightarrow{\pr_{2,L}} X_{L}\]
  is the morphism $\sigma:X_{L} \to X_{L}$. Then, $\pr_{1,L}^*\mc{E} \cong \pr_{2,L}^*\mc{E}$ implies that 
  \[\mc{E} \, \cong \, p_\sigma^*\pr_{1,L}^*\mc{E} \, \cong \, p_\sigma^*\pr_{2,L}^*\mc{E} \, \cong \, \sigma^*\mc{E}.\]
  
  Therefore, for any $\sigma \in G_L$, there exists an isomorphism
  $\lambda'_\sigma:\mc{E} \to \sigma^*\mc{E}$. Let $\tau \in G_L$. Since $\mc{E}$ is simple, 
  \[\mr{End}(\sigma^*\tau^*\mc{E}) \cong \sigma^*\tau^*\mr{End}(\mc{E}) \cong L.\]
  Hence, there exists $a_{\sigma,\tau} \in L^\times=\mr{Aut}(\sigma^*\tau^*\mc{E})$ such that the following diagram is commutative:
  \[\begin{diagram}
     \mc{E}&\rTo^{\lambda'_\sigma}&\sigma^*\mc{E}\\
     \dTo^{\lambda'_{\sigma\tau}} &\circlearrowleft &\dTo_{\sigma^*\lambda'_\tau}\\
     \sigma^*\tau^*\mc{E}&\rTo^{a_{\sigma,\tau}}&\sigma^*\tau^*\mc{E}
    \end{diagram}\]
 This directly implies the following equalities: Given $g_1, g_2, g_3 \in G_L$, we have
 \begin{eqnarray}
  a_{g_1,(g_2g_3)} \circ \lambda'_{g_1(g_2g_3)}&=&(g_1^*\lambda'_{g_2g_3}) \circ \lambda'_{g_1} \label{ceq1}\\
  g_1^*a_{g_2g_3} \circ g_1^*\lambda'_{g_2g_3}&=&(g_1^*g_2^*\lambda'_{g_3}) \circ g_1^*\lambda'_{g_2} \label{ceq2}\\
  a_{g_1,g_2} \circ \lambda'_{g_1g_2}&=&(g_1^*\lambda'_{g_2}) \circ \lambda'_{g_1} \label{ceq3}\\
  ((g_1g_2)^*\lambda'_{g_3}) \circ \lambda'_{g_1g_2}&=&a_{(g_1g_2),g_3}\circ \lambda'_{(g_1g_2)g_3} \label{ceq4} \\
   \end{eqnarray}
  Applying $g_1^*a_{g_2,g_3} \circ -$ to both sides of \eqref{ceq1}, we get,
  \begin{eqnarray*}
   g_1^*a_{g_2,g_3} \circ a_{g_1,(g_2g_3)} \circ \lambda'_{g_1(g_2g_3)}&=&g_1^*a_{g_2,g_3} \circ(g_1^*\lambda'_{g_2g_3}) \circ \lambda'_{g_1}\\
   &=& (g_1^*g_2^*\lambda'_{g_3}) \circ g_1^*\lambda'_{g_2}\circ \lambda'_{g_1} \, \, \, \,  \, \, \, \mbox{ by } \, \eqref{ceq2}\\
   &=& (g_1^*g_2^*\lambda'_{g_3}) \circ a_{g_1,g_2} \circ \lambda'_{g_1g_2}  \, \, \, \,   \mbox{ by } \, \eqref{ceq3}\\
   &=&a_{g_1,g_2} \circ a_{(g_1g_2),g_3} \circ \lambda'_{g_1g_2g_3} \, \, \, \, \mbox{ by }  \, \eqref{ceq4}
  \end{eqnarray*}
 where the last equality follows from the fact that multiplication by a scalar $a_{g_1,g_2}$ commutes with $(g_1^*g_2^*\lambda'_{g_3})$.
 Since $\lambda'_{g_1g_2g_3}$ is an isomorphism, we have the $2$-cocycle condition:
  \[g_1^*a_{g_2,g_3} \circ a_{g_1,(g_2g_3)}=a_{g_1,g_2} \circ a_{(g_1g_2),g_3}.\]
  Since $K$ is a $C_1$ field, $H^2(K,\mb{G}_m)=0$ (see \cite[p. $161$, Proposition $10$]{ser}). 
  This means that for any sequence $(a_{\sigma,\tau})_{\sigma,\tau \in G_L}$ satisfying
  the $2$-cocycle condition there exists a continuous morphism $\phi:G_L \to L^\times$ such that 
  \[a_{\sigma,\tau}=\sigma\phi(\tau)\phi(\sigma\tau)^{-1}\phi(\sigma).\]
 Consider now the isomorphism given by 
 \[\lambda_\sigma:=\phi(\sigma)^{-1}\lambda'_\sigma:\mc{E} \xrightarrow{\lambda'_\sigma} \sigma^*\mc{E} \xrightarrow{\phi(\sigma)^{-1}} \sigma^*\mc{E}, \, \mbox{ for all } \, \sigma \in G_L.\]
 Since $\phi(\sigma)$ is scalar, it commutes with $\lambda'_{\sigma\tau}$ i.e., we have the following commutative diagram:
 \[\begin{diagram}
  \mc{E}&\rTo^{\lambda'_{\sigma \tau}}&(\sigma\tau)^*\mc{E}\\
  \dTo^{\phi(\sigma)}&\circlearrowleft&\dTo_{\phi(\sigma)}\\
  \mc{E}&\rTo^{\lambda'_{\sigma \tau}}&(\sigma \tau)^*\mc{E}
       \end{diagram}\]
 Using \eqref{ceq3} and substituting for $a_{\sigma, \tau}$, we conclude that $\sigma^*\lambda'_{\tau} \circ \lambda'_{\sigma}$ equals
 \[a_{\sigma,\tau} \circ \lambda'_{\sigma \tau}=\sigma \phi(\tau) \phi(\sigma \tau)^{-1} \phi(\sigma)  \lambda'_{\sigma \tau}=\sigma \phi(\tau) \phi(\sigma \tau)^{-1} \lambda'_{\sigma \tau} \phi(\sigma).\]
 Therefore, \[(\sigma^* \lambda_{\tau})=(\sigma \phi(\tau))^{-1} \sigma^*\lambda'_\tau=(\phi(\sigma \tau)^{-1}  \lambda'_{\sigma \tau}) \circ (\phi(\sigma)^{-1} \lambda'_\sigma)^{-1}=\lambda_{\sigma \tau} \circ \lambda_\sigma^{-1}.\]
 Hence, $(\sigma^* \lambda_{\tau}) \circ \lambda_\sigma=\lambda_{\sigma \tau}$. 
 This proves the proposition.
 \end{proof}

 \vspace{0.2 cm} 
 
 \begin{thm}\label{elw7}
 Let $L$ be a finite Galois extension of $K$ and $f:X_L \to X$ the natural morphism. The natural functor from the category $\mathfrak{S}_X$ of simple sheaves on 
 $X$ to the category $\mathfrak{SK}_{X_L}$ of $K$-simple sheaves  on $X_L$, defined by pull-back of sheaves via $f$, is an equivalence of categories. 
 \end{thm}

 \begin{proof}
 It suffices to show that any simple sheaf $\mc{E}_L$ on $X_L$ is $K$-simple if and only if there exists a simple sheaf $\mc{E}$ on $X$ such that 
 $f^*\mc{E} \cong \mc{E}_L$. 
 By Proposition \ref{blrprop}, for any simple sheaf $\mc{E}$ on $X$, we have $f^*\mc{E}$ is $K$-simple on $X_{L}$.
  We now prove the converse. Let $\mc{E}_{L}$ be $K$-simple on $X_{L}$.  
  Proposition \ref{c8} implies that there exists a collection $(\lambda_\sigma)_{\sigma \in G_L}$ of isomorphisms 
  $\lambda_\sigma:\mc{E}_{L} \to \sigma^*\mc{E}_{L}$ such that $(\sigma^* \lambda_\tau)\circ \lambda_\sigma=\lambda_{\sigma \tau}$
  for any pair $\sigma, \tau \in G_L$. By Galois descent, this implies that there exists a coherent sheaf $\mc{E}$ on $X$ such that $\mc{E}_L \cong f^*\mc{E}$. 
  Since $f$ is flat, \[f^*(\mr{End}(\mc{E}))\cong \mr{End}(\mc{E}_L) \cong L.\] As $\mr{End}(\mc{E})$ is a $K$-vector space, this directly implies 
  that $\mr{End}(\mc{E}) \cong K$ i.e., $\mc{E}$ is simple.
   This proves the theorem.
 \end{proof}

\section{Moduli of $K$-geometrically stable sheaves}
Keep Notations \ref{e1}. In this section we use Theorem \ref{elw7} to prove that the set of $K$-rational points of the moduli space of geometrically stable sheaves on $X$ is in $1-1$ correspondence with the set 
of geometrically stable sheaves on $X$.

 Recall, the definition of the moduli functor of semi-stable sheaves over a projective variety $X$.
 
 \begin{defi}
 Fix $P$ the Hilbert polynomial of a coherent sheaf on $X$ with rank coprime to degree. 
 Denote by $\mc{M}_{X}(P)$ the moduli functor:
  \[\mc{M}_{X}(P) : \{\mr{Sch}/{K}\}^{\circ} \rightarrow \mr{Sets}\] 
such that for a $K$-scheme $T$,         
   \[ \mc{M}_{X}(P)(T):= \left\{ \begin{array}{l}
    \mbox{ isomorphism classes of } \mbox{ pure sheaves } \mc{F} \mbox{ on } X \times T \mbox{ flat }\\
    \mbox{ over } T \mbox{ and for every geometric point } t \in T, \  \mc{F}|_{X_t} \\
    \mbox{ is a stable sheaf with Hilbert polynomial } P \mbox{ on } X_{t}
    \end{array} \right\}/\sim \] 
  where $\mc{F} \sim \mc{G}$ if there exists an invertible sheaf $\mc{L}$ on $T$ such that $\mc{F} \cong \mc{G} \otimes p_T^*\mc{L}$, where 
  $p_T:X_T \to T$ is the natural projection map.
   \end{defi}

\begin{thm}[{\cite[Theorem $4.1$]{langmix}}]\label{langermod}
Let $R$ be a universally Japanese ring and $f : X \to S$ a projective morphism of $R$-schemes of finite type with geometrically connected fibres, and let $\mc{O}_X(1)$ be an $f$-ample line bundle. Then, for a fixed polynomial $P$, there exists a projective $S$-scheme $M_{X/S}(P)$ of finite type over $S$ which uniformly corepresents the functor

\[\mc{M}_{X/S}(P) : \{\mr{Sch}/{S}\}^{\circ} \rightarrow \mr{Sets}\] 
such that for a $S$-scheme $T$,         
\[ \mc{M}_{X/S}(P)(T):= \left\{ \begin{array}{l} 
 S \mbox{ equivalence classes of } \mbox{ pure sheaves } \mc{F} \mbox{ on } X \times_{S} T \mbox{ flat }\\
 \mbox{ over } T \mbox{ such that for every geometric point } t \in T, \ \mc{F}|_{X_t} \\ 
 \mbox{ is a semi-stable sheaf with Hilbert polynomial } P \mbox{ on } X_{t}
    \end{array} \right\}/\sim \] 
  where $\mc{F} \sim \mc{G}$ if there exists an invertible sheaf $\mc{L}$ on $T$ such that $\mc{F} \cong \mc{G} \otimes p_T^*\mc{L}$, where 
  $p_T:X\times_{S}T \to T$ is the natural projection map.
  
Moreover, there is an open subscheme $M^s_{X/S}(P)$ of $M_{X/S}(P)$ which universally corepresents the subfunctor of families of geometrically stable sheaves.
\end{thm} 
 
\begin{rem}
Note that a $C_1$-field $K$ is a universally Japanese ring. If the rank and degree of coherent sheaves with Hilbert polynomial $P$, are coprime, then stability and semi-stability coincide (see \cite[Lemmas $1.2.13$ and $1.2.14$]{huy}). 
Then by Theorem \ref{langermod}, there exists a projective $K$-scheme of finite type $M_{X}(P)$, universally corepresenting the functor $\mc{M}_{X}(P)$.
 \end{rem}
  
We now review briefly the construction of the moduli scheme $M_X(P)$.
By \cite[Theorem $4.2$]{langmix}, there exists an integer $e$ such that any semi-stable sheaf on $X$ with Hilbert polynomial $P$ is $e$-regular (in the sense of Castelnuovo-Mumford regularity). 
 Fix such an integer $e$. Denote by $\mc{H}:=\mo_{X}(-e)^{\oplus P(e)}$ and by $\mr{Quot}_{\mc{H}/X/P}$ the Quot scheme parametrizing all quotients of the form 
$\mc{H} \twoheadrightarrow \mc{Q}_0$, where $\mc{Q}_0$ has Hilbert polynomial $P$ (see \cite[\S $4.4$]{S1} for more details).

 Let $\mc{R}$ be the subset of $\mr{Quot}_{\mc{H}/X/P}$ consisting of all points which parametrize 
 quotients of the form $\mc{H}\twoheadrightarrow \mc{Q}_0$ such that $\mc{Q}_0$ is semi-stable and 
 $H^0(\mc{Q}_0(e))$ is (non-canonically) isomorphic to $k^{\oplus P(e)}$.
   Now, semi-stability is an open condition (see \cite[Proposition $2.3.1$]{huy}). Therefore, $\mc{R}$ is an open subscheme in $\mr{Quot}_{\mc{H}/X/P}$.   
   The group $\mr{GL}(P(e))=\mr{Aut}(\mc{H})$ acts on $\mr{Quot}_{\mc{H}/X/P}$ from the right by the composition $[\rho] \circ g=[\rho \circ g]$, where 
      $[\rho:\mc{H} \to \mc{F}] \in \mr{Quot}_{\mc{H}/X/P} \mbox{ and } g \in \mr{GL}(P(e))$. By \cite[Theorem $4.3$]{langmix}, $\mc{R}$ is the set of semi-stable points of $\mr{Quot}_{\mc{H}/X/P}$
   under this group action. The moduli scheme 
   $M_X(P)$ of semi-stable sheaves on $X$ with Hilbert polynomial $P$ is the geometric quotient of $\mc{R}$ 
   under this action (see \cite[pp. $582$, after Theorem $4.3$]{langmix}). Denote by 
   \begin{equation}\label{cn14}
    \pi:\mc{R} \to M_X(P)
   \end{equation}
 the corresponding quotient morphism. The quotient exists due to Seshadri's result \cite[Theorem $4$]{sesred}.

\begin{thm}\label{elw9}
 There is a $1-1$ correspondence between the set of $K$-rational points of $M_X(P)$ and the set of 
isomorphism classes of geometrically stable sheaves on $X$ with Hilbert polynomial $P$. 
\end{thm}

\begin{proof}
Let $x:\Spec(K) \to M_X(P)$ be a $K$-rational point of $M_X(P)$. Denote by $\mc{R}_K$ the base change of the morphism 
$\pi$ in \eqref{cn14} by the morphism $x$. Let $y \in \mc{R}_K$ be a closed point. Then $y$ is a $L$-rational 
point on $\mc{R}_K$ for some finite extension $L$ of $K$. Without loss of generality assume that $L$ is a finite 
Galois extension of $K$. Since the Quot functor is representable, the point $y$ corresponds to a quotient $\phi_y:\mc{H}_L \twoheadrightarrow \mc{E}_L$ on $X_L:=X \times_K \Spec(L)$, where 
$\mc{H}_L \cong \mc{H} \otimes_K L$ and $\mc{E}_L$ is geometrically stable with Hilbert polynomial $P$ on $X_L$.
For any $\sigma \in \mr{Gal}(L/K)$ and the induced morphism $f_\sigma:X_L \xrightarrow{\mr{id}_X \times \sigma} X_L$, denote by $f_\sigma^*\phi_y$ the quotient morphism,
\[\mc{H}_L \cong f_\sigma^*\mc{H}_L \twoheadrightarrow f_\sigma^*\mc{E}_L.\]
By the universal property of Quot scheme, $f_\sigma^*\phi_y$ corresponds to the composed morphism:
\[\Spec(L) \xrightarrow{\sigma} \Spec(L) \xrightarrow{y} \mc{R}_K.\]
Since $\mc{R}_K$ is a fiber to the morphism $\pi$, this implies $f_\sigma^*\mc{E}_L \cong \mc{E}_L$.
As a result we obtain a set $(\lambda_\sigma)_{\sigma \in \mr{Gal}(L/K)}$ of isomorphisms $\lambda_\sigma:\mc{E}_L \to \sigma^*\mc{E}_L$.

Consider now the morphism 
 \[\Phi:L \otimes_K L \to \coprod\limits_{\sigma \in \mr{Gal}(L/K)} L, \, \mbox{ defined by } \, a \otimes b \mapsto (a\sigma(b))_{\sigma \in \mr{Gal}(L/K)}.\]
 It is easy to check that $\Phi$ is an isomorphism. Consider now the induced morphisms,
 \[\Phi_i:\coprod\limits_{\sigma \in \mr{Gal}(L/K)} X_L \xrightarrow{\Phi} X_{L \otimes_K L} \xrightarrow{\pr_{i,L}} X_L \, \mbox{ for } i=1,2.\]
 Observe that \[\Phi_1=\coprod \mr{id}\, \mbox{ and } \,\Phi_2=\coprod\limits_{\sigma \in \mr{Gal}(L/K)} f_\sigma,\]
 where $f_\sigma:X_L \xrightarrow{\mr{id}_X \times \sigma} X_L$. The set of isomorphims $(\lambda_\sigma)_{\sigma \in \mr{Gal}(L/K)}$
 then induce an isomorphism $\psi:\Phi_1^*\mc{E}_L \xrightarrow{\sim} \Phi_2^*\mc{E}_L$. Since $\Phi$ is an isomorphism, $\psi$ induces 
 an isomorphism $\pr_{1,L}^*\mc{E}_L \xrightarrow{\sim}\pr_{2,L}^*\mc{E}_L$ i.e., $\mc{E}_L$ is $K$-geometrically stable.
 
 Since $\mc{E}_L$ is geometrically stable, $\mc{E}_{\ov{K}}:=\mc{E}_L \otimes_L \ov{K}$ is stable. By Lemma  \ref{stabsimp}, $\mc{E}_{\ov{K}}$
 is simple. Since $\mr{End}(\mc{E}_L)$ is an $L$-vector space and 
 \[\mr{End}(\mc{E}_L) \otimes_L \ov{K} \cong \mr{End}(\mc{E}_{\ov{K}}) \cong \ov{K}\]
 we conclude that $\mr{End}(\mc{E}_L) \cong L$ i.e., $\mc{E}_L$ is simple. In particular, $\mc{E}_L$ is $K$-simple.
 Using Theorem \ref{elw7}, there exists a simple sheaf $\mc{E}$ on $X$ such that $\mc{E} \otimes_K L \cong \mc{E}_L$. 
 By \cite[Theorem $1.3.7$]{huy}, it follows that $\mc{E}$ is geometrically stable.
 
 Conversely, by Theorem \ref{langermod}, any geometrically stable sheaf on $X$ with Hilbert polynomial $P$ corresponds to a
 $K$-rational point on $M_X(P)$. It is easy to check that the geometrically stable sheaf $\mc{E}$ corresponds to the $K$-rational point $x$ of 
 $M_X(P)$. This gives us a $1-1$ correspondence between the $K$-rational points of $M_X(P)$ and the 
 set of geometrically stable sheaves on $X$ with Hilbert polynomial $P$.
 This proves the theorem.
 \end{proof}
 
 \section{Applications of descent theory}\label{sec4}
 
 In this section, we mention the application of descent theory studied before.
 Recall, the definition of a $C_1$ field given in the introduction.
 
\begin{exa} We state without proof some examples of $C_1$ fields:
\begin{enumerate}
 \item An algebraically closed field is trivially $C_{1}$.
 \item Finite fields are $C_{1}$ (see \cite{C}).
 \item The function field of an irreducible curve defined over an algebraically closed field is $C_{1}$(see \cite{T}).
 \item Let $R$ be a Henselian discrete valuation ring of characteristic $0$ with residue field denoted $k$, of characteristic $p$ and fraction field denoted $K$.
If $k$ is algebraically closed, then $K$ is $C_1$ (see \cite[Theorem $14$]{lang1}).
\end{enumerate}
\end{exa}

\begin{defi}
A variety $Y$ over an algebraically closed field $\ov{K}$ is \emph{separably rationally connected}\index{rationally connected! separably} if there exists a morphism $f: \mathbf{P}^{1} \to Y$ such that $f^{*}(T_Y)$ is ample.
\end{defi}

\begin{rem}\label{rcsrc} 
Note that over an algebraically closed field $\ov{K}$ of characteristic $0$, rationally connected is equivalent to separably rationally connected (see \cite[Proposition $IV.3.3.1$]{K}).
\end{rem}

\noindent
\textbf{The $C_{1}$ conjecture} (Lang-Manin-Koll\'ar):
A smooth, proper, separably rationally connected variety over a $C_1$ field always has a rational point.

The conjecture has already been proven for various $C_{1}$ fields (see \cite[Chapter $2$]{ind} for a complete list).

\begin{rem}
The conjecture remains open in the case when the $C_{1}$ field is the fraction field of a maximal unramified discrete valuation ring with algebraically closed residue field of mixed characteristic.

Recently, the conjecture was shown to hold trivially for certain rationally connected varieties 
over such fields (see \cite{ind2}).
Let $M^s_{X_{K},\mc{L}_{K}}(r,d)$ be the moduli space of geometrically stable locally free sheaves of rank $r$ and determinant 
$\mc{L}_K$. Denote by  $M^s_{X_{\ov{K}},\mc{L}_{\ov{K}}}(r,d)$ 
the moduli space of geometrically stable locally free sheaves of  rank $r$ and determinant $\mc{L}_{\ov{K}}:= \mc{L}_K \otimes_{K} \ov{K}$ over the curve $X_{\ov{K}}:= X_K \times_K \msp(\ov{K})$. 
By \cite{sesh}, $M^s_{X_{\ov{K}},\mc{L}_{\ov{K}}}(r,d)$ is a unirational variety and therefore rationally connected. Since the moduli space $M^s_{X_{\ov{K}},\mc{L}_{\ov{K}}}(r,d)$ is the base change 
$M^s_{X_{K},\mc{L}_{K}}(r,d) \times_{K} \msp(\ov{K})$, this implies $M^s_{X_{K},\mc{L}_{K}}(r,d)$ is 
rationally connected. 
Suppose that $K$ is the fraction field of a Henselian discrete valuation ring with algebraically closed residue field.
Then, $M^s_{X_{K},\mc{L}_{K}}(r,d)$ has a $K$-rational point.
This is shown using descent theory of stable sheaves on smooth, projective curves (\cite[Theorem $1.2$]{ind2}).
It is natural to ask how the descent theory of simple sheaves over $C_1$ fields can be used to prove the existence of rational points, in the
higher dimension case. We cite some recent results in this direction.
\end{rem}

 \begin{defi} 
  Recall, the \emph{index} of $X$, denoted $\mr{ind}(X)$, is the gcd of the set of degrees of zero dimensional cycles on $X$.
 \end{defi}
 
 \begin{rem}
  Clearly, the notion of index generalizes the idea of having a rational point. Applying Theorem \ref{elw7} to the case of 
  invertible sheaves on varieties, one can obtain a sufficient criterion for any projective variety defined over a $C_1$ field to have index $1$.
 \end{rem}

\begin{defi}
  Denote by $G$ the absolute Galois group $\mr{Gal}(\ov{K}/K)$. An invertible sheaf $\mc{L}_{\ov{K}}$ on $X_{\ov{K}}:=X \times_K \Spec(\ov{K})$ is called $G$-\emph{invariant}
  if for any $\sigma \in G$ and the corresponding morphism $\sigma:X_{\ov{K}} \to X_{\ov{K}}$, we have $\sigma^*\mc{L}_{\ov{K}} \cong \mc{L}_{\ov{K}}$.
  Denote by $\Lambda \subset \mr{Pic}(X_{\ov{K}})$ the subgroup of $\mr{Pic}(X_{\ov{K}})$ consisting of all $G$-invariant invertible sheaves on $X_{\ov{K}}$.
  Denote by $e:=\mr{gcd}\{\chi(\mc{L}_{\ov{K}})|\mc{L}_{\ov{K}} \in \Lambda\}$. We call $e$ the \emph{linear index of} $X$, denoted $\mr{lin-ind}(X)$.
 \end{defi}
 

 \begin{thm}\label{thm1}
 Suppose that $H^1(\mathcal{O}_X)=0$, $\mr{Pic}(X_{\ov{K}})$ is of rank $r$, generated by $\mc{L}_1, ..., \mc{L}_{r-1}$ and  $\mc{L}_r:=H_{\ov{K}}=H \otimes_K \ov{K}$
   satisfying the following conditions:
   \begin{enumerate}
    \item the ideal $(\deg(\mc{L}_1),\deg(\mc{L}_2),...,\deg(\mc{L}_r))$ in $\mb{Z}$ generated by $\deg(\mc{L}_i)$ for $i=1,...,r$ coincides with the ideal $(1)$,
    \item for any $r \times r$-matrix $A=(a_{i,j})$ with integral entries $a_{i,j}$, $a_{r,k}=0$ for all $k<r$, $a_{r,r}=1$, $A \not=\mr{Id}$ and $A^t=\mr{Id}$ for some $t>0$, 
    we have $\sum_j a_{ij} \deg(\mc{L}_j)\not= \deg(\mc{L}_i)$ for some $i>0$. 
   \end{enumerate}
  Then, each $\mc{L}_i$ is $G$-invariant, $\mr{lin-ind}(X)=\mr{gcd}\{\chi(\mc{L}_i(n))| i=1,...,r \mbox{ and } n \in \mb{Z}\}=1$ and 
  \[\mr{ind}(X)=1 \, \mbox{ if } \mr{char}(k)=0 \mbox{ and } \mbox{ prime-to-p part of } \mr{ind}(X) \mbox{ equals } 1 \mbox{ if } \mr{char}(k)=p>0.\]
   \end{thm}
By prime-to-p part of $N$ we mean the largest divisor of $N$ which is prime to $p$. 

\begin{proof}
 See \cite{indmad} for the proof.
\end{proof}

As a consequence of Theorem \ref{thm1}, we obtain numerous examples of smooth, projective varieties on $C_1$-fields with index $1$.  
\begin{exa}\label{elw8}
Let $X$ be a smooth, projective variety with $\deg(H_{\ov{K}})>2$, $H^1(\mathcal{O}_X)=0$, $\mr{Pic}(X_{\ov{K}})$ is of rank $2$ and there exists an invertible sheaf $\mc{L}_0$ of degree coprime to $\deg(H_{\ov{K}})$
(for example, a smooth surface $X$ in $\mb{P}^3_K$ of degree at least $3$ with $\mr{rk}(\mr{Pic}(X_{\ov{K}}))=2$ and $X_{\ov{K}}$ contains a curve of degree coprime to $\deg(X)$).
Theorem \ref{thm1} implies that every invertible sheaf on $X_{\ov{K}}$ is $G$-invariant and $\mr{ind}(X)=\mr{lin-ind}(X)=1.$
\end{exa}


\begin{thebibliography}{10}

\bibitem{bn}
S.~Bosch, W.~L{\"u}tkebohmert, and M.~Raynaud.
\newblock {\em N{\'e}ron models}, volume~21.
\newblock Springer Science \& Business Media, 2012.

\bibitem{C}
C.~{Chevalley}.
\newblock {D\'emonstration d'une hypoth\`ese de M. {A}rtin.}
\newblock {\em {Abh. Math. Semin. Univ. Hamb.}}, 11:73--75, 1935.

\bibitem{indmad}
A.~Dan and I.~Kaur.
\newblock Examples of varieties with index one on {$C_1$}-fields.
\newblock {\em Journal of Number Theory}, 203:242--248, 2019.

\bibitem{gill}
P.~Gille and T.~Szamuely.
\newblock {\em Central simple algebras and Galois cohomology}, volume 165.
\newblock Cambridge University Press, 2017.

\bibitem{GHS}
T.~Graber, J.~Harris, and J.~Starr.
\newblock Families of rationally connected varieties.
\newblock {\em Journal of the American Mathematical Society}, pages 57--67,
  2003.

\bibitem{G}
M.~J. Greenberg.
\newblock {\em Lectures on forms in many variables}, volume~31.
\newblock W. A. Benjamin New York-Amsterdam, 1969.

\bibitem{HL}
B.~Hassett and K.~Lai.
\newblock Cremona transformations and derived equivalences of K3 surfaces.
\newblock {\em Compositio Mathematica}, 154(7):1508--1533, 2018.

\bibitem{huy}
D.~Huybrechts and M.~Lehn.
\newblock {\em The geometry of moduli spaces of sheaves}.
\newblock Springer, 2010.

\bibitem{ind}
I.~Kaur.
\newblock {\em The ${C_1}$ conjecture for the moduli space of stable vector
  bundles with fixed determinant on a smooth projective curve}.
\newblock Ph. d. thesis, Freie University Berlin, 2016.

\bibitem{ind2}
I.~Kaur.
\newblock A pathological case of the {$C_1$} conjecture in mixed
  characteristic.
\newblock {\em Mathematical Proceedings of the Cambridge Philosophical
  Society}, 167(1):61--64, 2019.

\bibitem{K}
J.~Koll{\'a}r.
\newblock {\em Rational curves on algebraic varieties}, volume~32.
\newblock Springer Science \& Business Media, 2013.

\bibitem{KL}
D.~Krashen and M.~Lieblich.
\newblock Index reduction for brauer classes via stable sheaves.
\newblock {\em International Mathematics Research Notices}, 2008.

\bibitem{lang1}
S.~Lang.
\newblock On quasi algebraic closure.
\newblock {\em Annals of Mathematics}, pages 373--390, 1952.

\bibitem{langmix}
A.~Langer.
\newblock Moduli spaces of sheaves in mixed characteristic.
\newblock {\em Duke Mathematical Journal}, 124(3):571--586, 2004.

\bibitem{S1}
E.~Sernesi.
\newblock {\em Deformations of Algebraic Schemes}.
\newblock Grundlehren der Mathematischen Wissenschaften-334. Springer-Verlag,
  2006.

\bibitem{ser}
J.~P. Serre.
\newblock {\em Local fields}, volume~67.
\newblock Springer Science \& Business Media, 2013.

\bibitem{sesred}
C.~S. Seshadri.
\newblock Geometric reductivity over arbitrary base.
\newblock {\em Advances in Mathematics}, 26(3):225--274, 1977.

\bibitem{sesh}
C.~S. Seshadri.
\newblock {\em Fibr{\'e}s vectoriels sur les courbes alg{\'e}briques:
  conf{\'e}rences {\`a} l'ENS, Juin 1980}.
\newblock Number 95-96. Soci{\'e}t{\'e} Math{\'e}matique de France, 1982.

\bibitem{T}
C.~{Tsen}.
\newblock {Divisionsalgebren \"uber Funktionenk\"orpern.}
\newblock {\em {Nachr. Ges. Wiss. G\"ottingen, Math.-Phys. Kl.}},
  1933:335--339, 1933.

\end{thebibliography}
\end{document}